\documentclass[11pt,reqno]{amsart}
\usepackage{tikz,a4wide}
\usetikzlibrary{arrows.meta,math,calc}
\usepackage{tikz-3dplot}
\usepackage{amsmath,amssymb,amsthm,mathtools}
\usepackage{enumitem}
\usepackage[pagebackref,citecolor=black,linkcolor=blue,colorlinks=true]{hyperref}

\newtheorem{theorem}{Theorem}[section]
\newtheorem{lemma}[theorem]{Lemma}
\newtheorem{proposition}[theorem]{Proposition}
\newtheorem{corollary}[theorem]{Corollary}

\theoremstyle{definition}
\newtheorem{definition}[theorem]{Definition}

\numberwithin{equation}{section}

\newcommand{\R}{\mathbb{R}}
\newcommand{\Sf}{\mathbb{S}^2}
\newcommand{\Ha}{\mathcal{H}}
\newcommand{\N}{\mathcal{N}}
\newcommand{\dist}{\mathrm{dist}}
\newcommand{\Lip}{\mathrm{Lip}}
\newcommand{\Cone}{\mathrm{Cone}}
\newcommand{\M}{\mathcal{M}}
\newcommand{\eps}{\varepsilon}
\newcommand{\ssubset}{\subset\joinrel\subset}

\title{Boundary Plateau laws II}
\author[C. De Lellis]{Camillo De Lellis}
\address{School of Mathematics, Institute for Advanced Study, 1 Einstein Dr., Princeton NJ 08540, USA}
\address{Gran Sasso Science Institute, viale Francesco Crispi, 7, 67100 L'Aquila, Italy}
\email{camillo.delellis@ias.edu}
\author[F. Glaudo]{Federico Glaudo}
\email{federico.glaudo@gmail.com}

\begin{document}

\begin{abstract}
Given a straight line $\ell \subset \mathbb R^3$, we classify $2$-dimensional cones in $\mathbb R^3$ which contain $\ell$ and are Almgren minimizing with respect to $\ell$. The main idea of this paper is due to an LLM, and the paper was written by an LLM, under the guidance of the authors.
\end{abstract}

\maketitle

\section{Introduction}

This paper is a sequel to \cite{DLG}, to which we refer the reader for a more thorough introduction. Here we only give the mathematics and omit the physical and historical context. Moreover, our bibliography is far from complete, and we refer the reader to the bibliography of \cite{DLG}. 

Parts of this introduction are identical to the corresponding parts of \cite{DLG}. Our main result, Theorem \ref{t:main}, is stronger than the main result of \cite{DLG}: it shows that the cones over the $1$-skeleton of a cube and the ``thin fish'', which appear in the list of \cite{DLG}, are not minimizing. We refer to \cite{DLG} for a comparison of the proofs of the two results and in particular to \cite[Section 1.3]{DLG} for how to interpret the theorems in the context of understanding the shape of soap films attached to a wire at a point on the wire.

Incidentally, a shortened version of the arguments of this note can also be used in the interior, covering an important part of \cite{Taylor}.

\subsection{Main theorem}

\begin{definition}[Almgren minimizing sets]\label{d:Almgren-minimizing}
Let $k<N$ be two given positive integers, $U\subset \mathbb R^N$ an open subset, and $\Gamma$, $E$ two relatively closed subsets of $U$. $E$ is a $k$-dimensional Almgren minimizing set with respect to $\Gamma$ if $E$ has locally finite $\mathcal{H}^k$ measure in $U$ and the inequality
\begin{equation}\label{e:Almgren-minimizing}
\Ha^k (E\cap V) \leq \Ha^k (\Phi (E) \cap V)
\end{equation}
holds for every open set $V\ssubset U$ and every Lipschitz map $\Phi$ with
\[
\overline{W\cup \Phi (W)} \ssubset V\setminus \Gamma\, , \qquad \mbox{where $W= \{x: x\neq \Phi (x)\}$.}
\]
\end{definition}

The Almgren minimality property is invariant under addition of an $\mathcal{H}^k$-null set. As in \cite{DLG} we therefore impose on $E$ the normalization
\begin{equation}\label{e:normalization}
\mathcal{H}^k (E\cap B_r (x)) > 0 \qquad \forall x\in E \quad\mbox{and} \quad \forall r >0 \, .
\end{equation}
By a theorem of Almgren such sets are (countably) rectifiable away from $\Gamma$, cf. \cite[Theorem II.3, Conclusion 9]{Almgren-1976}; in our case $\Gamma$ is a line. In the rest of the paper ``minimizing'' refers to a rectifiable set as in Definition \ref{d:Almgren-minimizing} which satisfies \eqref{e:normalization}.

The focus of this paper is the classification of all $2$-dimensional cones which are minimizing with respect to a line $\ell$ contained in them.

\begin{theorem}\label{t:main}
Let $\ell\subset\mathbb R^3$ be a line through the origin and let $E\supset\ell$ be a closed cone which is minimizing with respect to $\ell$. Then, up to a rotation, $E$ is one of the following:
\begin{itemize}
\item[(a)] a half-plane with boundary $\ell$;
\item[(b)] two half-planes with boundary $\ell$, forming a convex angle $\geq\frac{2\pi}{3}$;
\item[(c)] three half-planes with boundary $\ell$, forming angles of $\frac{2\pi}{3}$;
\item[(d)] the cone over the $1$-skeleton of a regular tetrahedron, with $\ell$ one of the three lines it contains, i.e. the line through the midpoints of two opposite edges.
\end{itemize}
Conversely, each of the cones (a)--(d) is minimizing with respect to $\ell$.
\end{theorem}

\begin{figure}
\tdplotsetmaincoords{80}{15}
\begin{tikzpicture}[tdplot_main_coords]
\draw[thick] (-9,0,-1) -- (-8,0,-1) -- (-8,0,1) -- (-9,0,1);
\draw[thick, color=red] (-9,0,-1.2) -- (-9,0,1.2);
\draw[thick] (-6,0,-1) -- (-5,0,-1) -- (-5,0,1) -- (-6,0,1);
\draw[thick, color=red] (-6,0,-1.2) -- (-6,0,1.2);
\draw[thick] (-6,0,-1) -- ({-6-sqrt(2)*cos(50)}, {-sqrt(2)*sin(50)},-1)
-- ({-6-sqrt(2)*cos(50)}, {-sqrt(2)*sin(50)},1) -- (-6,0,1);
\draw[thick, color=red] (-3,1,-1.2) -- (-3,1,1.2);
\draw[thick] (-4,1,-1) -- (-3,1,-1);
\draw[thick] (-3,1,1) -- (-4,1,1) -- (-4,1,-1);
\draw[thick] (-3,1,1) -- (-2.5,{sqrt(3)/2+1},1) -- (-2.5,{sqrt(3)/2+1},-1);
\draw[thick] (-2.5,{sqrt(3)/2+1},-1) -- (-0.35*2.5 -0.65*3,{0.35*(sqrt(3)/2 + 1) + 0.65},-1);
\draw[thick, dashed] (-3,1,-1) -- (-0.35*2.5 -0.65*3,{0.35*(sqrt(3)/2 + 1) + 0.65},-1);
\draw[thick] (-3,1,1) -- (-2.5,{-sqrt(3)/2+1},1) -- (-2.5,{-sqrt(3)/2+1},-1) -- (-3,1,-1);
\end{tikzpicture}
\caption{The local boundary geometries (a), (b), and (c). The line $\ell$ is in red.}\label{f:PlateauB1}
\end{figure}
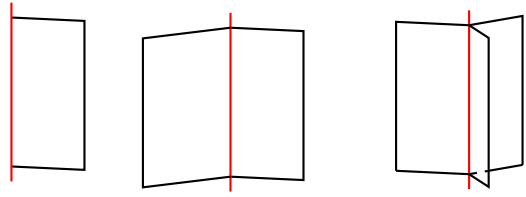

\begin{figure}
\tdplotsetmaincoords{80}{20}
\begin{tikzpicture}[tdplot_main_coords]
\draw[thick] (0,0, {2*sqrt(2)}) -- ({-sqrt(3)},-1,0) -- ({sqrt(3)}, -1, 0) -- (0,0,{2*sqrt(2)});
\draw[thick, dashed] ({sqrt(3)},-1,0) -- (0,2,0) -- (0,0,{2*sqrt(2)});
\draw[thick, dashed] ({-sqrt(3)},-1,0) -- (0,2,0);
\draw[dashdotted] ({-sqrt(3)},-1,0) -- (0,0,{sqrt(2)/2});
\draw[dashdotted] ({sqrt(3)},-1,0) -- (0,0,{sqrt(2)/2});
\draw[dashdotted] (0,2,0) -- (0,0,{sqrt(2)/2});
\draw[dashdotted] (0,0,{2*sqrt(2)}) -- (0,0,{sqrt(2)/2});
\draw[thick, red] (0,-1.2,{-0.1*sqrt(2)}) -- (0,1.2,{1.1*sqrt(2)});
\end{tikzpicture}
\caption{The local boundary geometry (d). The line $\ell$ is going through the midpoints of two opposite edges of the $1$-skeleton of the tetrahedron.}\label{f:PlateauB2}
\end{figure}
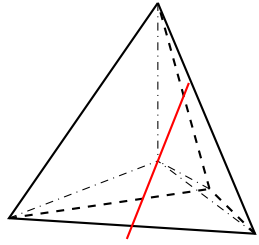

\subsection{The main idea}
It is standard to show that the link $E\cap\partial B_1$ of a minimizing cone $E$ as in Theorem \ref{t:main} is a stationary geodesic net on $\Sf$, i.e. a finite union of geodesic arcs; in particular, at every vertex which does not lie on $\ell$ exactly three arcs meet, forming angles of $\frac{2\pi}{3}$ (see Proposition \ref{p:stationary} below).

The key observation of this paper is the following: if $E$ is minimizing, then it must remain minimizing if we replace a convex face of its link whose angles are all equal to $\frac{2\pi}{3}$ with its medial axis. Informally, the medial axis of a polygon $F$ is the set of points of $F$ which have at least two closest points on $\partial F$; for a convex polygon it is a tree joining the vertices of $F$ (see Figure \ref{f:medial}).

More precisely, let $F$ be such a face. We move each planar sector of the cone over $\partial F$ along its inward normal until it reaches the cone over the medial axis of $F$. Two sectors arriving at the same piece of the medial cone are mapped onto the same set, and the angle condition guarantees that this does not increase the area. With a logarithmic cutoff, the deformation can be localized at an arbitrarily small cost in area, and the resulting set coincides with the cone over the medial axis of $F$ in a neighborhood of the origin. By Proposition \ref{p:halfspace} below, this gives an admissible competitor at an arbitrarily small additional cost. If $F$ is a quadrilateral or a pentagon, the new set either has strictly less area, or close to some point it consists of $k\geq 4$ half-planes bounded by a line, which is not minimizing. Hence such faces do not exist. This simplifies greatly the classification.

\subsection{Sketch of the proof}
We recall the following terminology from \cite[Definitions 4.1 and 4.2]{DLG}. A geodesic net $\N\subset\Sf$ is a finite union of geodesic arcs (the edges), which can intersect only at their endpoints (the vertices). $\N$ is a {\em stationary net with two antipodal pins} if two antipodal vertices $P$ and $-P$ are marked as {\em pins} and
\begin{itemize}
\item every vertex which is not a pin is the endpoint of exactly three edges;
\item every pin is the endpoint of one, two, or three edges;
\item any two edges with a common endpoint form an angle of at least $\frac{2\pi}{3}$.
\end{itemize}
In particular, at a vertex which is not a pin the three edges form angles of $\frac{2\pi}{3}$. The {\em degree} of a vertex is the number of edges ending at it. A {\em face} of $\N$ is a connected component of $\Sf\setminus \N$. If $F$ is a face and $V$ is a vertex of $\N$ lying on $\partial F$, the {\em angles of $F$ at $V$} are the openings of the sectors of $F$ at $V$ bounded by two consecutive edges ending at $V$ (if $V$ is a pin of degree $1$, the angle is $2\pi$). We call a face {\em equiangular} if all its angles, including those at the pins, are equal to $\frac{2\pi}{3}$.

The starting point is the following proposition from \cite{DLG}.

\begin{proposition}[{\cite[Proposition 4.3]{DLG}}]\label{p:stationary}
Assume $E\subset \mathbb R^3$ is a cone containing a line $\ell$ which is minimizing with respect to it. The link of $E$, namely $E\cap \partial B_1$, is a stationary net with two antipodal pins, which are the two points forming the set $\ell \cap \partial B_1$.
\end{proposition}

The map $\Phi$ which produces a competitor in Definition \ref{d:Almgren-minimizing} is constrained to be the identity in a neighborhood of $\ell$. The next proposition weakens this constraint: it allows maps which touch $\ell$, as long as they do not move anything in a half-space bounded by a plane containing $\ell$. It is proved in Section \ref{s:relaxed}.

\begin{proposition}[Half-space relaxed condition]\label{p:halfspace}
Assume that $\ell\subset\mathbb R^3$ is a line through the origin and that $E\subset\mathbb R^3$ is the cone over a connected geodesic net $\N\subset\Sf$ whose pins are the two points of $\ell\cap\partial B_1$. Let $\Psi:E\to\mathbb R^3$ be a Lipschitz map such that:
\begin{itemize}
\item[(i)] $\Psi(x)=x$ outside $B_1$;
\item[(ii)] $\Psi(x)=x$ for every $x\in\ell$;
\item[(iii)] for some open half-space $H_+$ with $\ell\subset\partial H_+$, we have $\Psi(x)=x$ for every $x\in E\cap H_+$ and $\Psi^{-1}(H_+)=E\cap H_+$.
\end{itemize}
Then, for every $\varepsilon>0$, there is a Lipschitz map $G:\mathbb R^3\to\mathbb R^3$ such that:
\begin{itemize}
\item[(a)] $G$ is the identity outside a compact set, and $G(x)=x$ for every $x\in E\setminus B_1$;
\item[(b)] upon setting $W:=\{x:\ G(x)\neq x\}$, we have $\overline{W\cup G(W)}\subset\mathbb R^3\setminus\ell$;
\item[(c)] $\Ha^2\big(G(E\cap\overline B_1)\big)\le\Ha^2\big(\Psi(E\cap\overline B_1)\big)+\varepsilon$.
\end{itemize}
\end{proposition}

\begin{corollary}\label{c:halfspace}
If $E$ is minimizing with respect to $\ell$, then $\Ha^2(E\cap\overline B_1)\le\Ha^2\big(\Psi(E\cap\overline B_1)\big)$ for every $\Psi$ as in Proposition \ref{p:halfspace}.
\end{corollary}

The following proposition is proved in Section \ref{s:medial}.

\begin{proposition}[Convex face lemma]\label{p:convex-face}
Let $F\subset\Sf$ be a convex spherical polygon with all angles equal to $\frac{2\pi}{3}$. Let $\Sigma$ be the cone over the sides of $F$. Let $\M(F)$ be the medial axis of $F$: the set of points of $F$ whose distance to $\partial F$ is attained on at least two different sides. Let $\Sigma':=\Cone(\M(F))$.

For every $\delta>0$ there are $R>1$ and a Lipschitz map $\Phi:\Sigma\to\mathbb R^3$ such that:
\begin{itemize}
\item[(1)] $\Phi(x)=x$ for every $x$ on the half-lines through the vertices of $F$;
\item[(2)] $\Phi(x)=x$ for every $x\in\Sigma\setminus B_R$, and $\Phi(\Sigma\cap B_R)\subset B_R$;
\item[(3)] $\Phi(\Sigma)\subset \Cone(F)$;
\item[(4)] $\Phi(\Sigma)\cap B_1=\Sigma'\cap B_1$;
\item[(5)] $\Ha^2\big(\Phi(\Sigma\cap B_R)\big)\le\Ha^2(\Sigma\cap B_R)-\eta(F)+\delta$.
\end{itemize}
Here $\eta(F)\geq 0$ is defined as follows. Let $\sigma_1,\dots,\sigma_k$ be the sides of $F$, and let $n_i$ be the unit normal to the plane containing $\sigma_i$, pointing inside $F$. For $i\neq j$, let $\Sigma'_{ij}$ be the cone over the set of points of $\M(F)$ whose distance to $\partial F$ is attained on both $\sigma_i$ and $\sigma_j$. Then
\[
\eta(F):=\sum_{i<j}\big(|n_i-n_j|-1\big)\,\Ha^2(\Sigma'_{ij}\cap B_1)\, .
\]
Every term of this sum is nonnegative, and the terms with $\sigma_i,\sigma_j$ adjacent vanish. Moreover, $\eta(F)>0$ if and only if no point of $F$ is equidistant from all the sides of $F$.
\end{proposition}

Combining the last two propositions we prove in Section \ref{s:faces} the following proposition.

\begin{proposition}[Faces with angles $\frac{2\pi}{3}$]\label{t:faces}
Let $\ell\subset\mathbb R^3$ be a line through the origin, and let $E\supset\ell$ be a cone which is minimizing with respect to $\ell$. Then every equiangular face of the link $E\cap\partial B_1$ is a biangle or a triangle.
\end{proposition}

The idea is the following. Let $F$ be an equiangular face with $k\in \{4,5\}$ sides. We deform the cone over $\partial F$ with the map $\Phi$ of Proposition \ref{p:convex-face}, leaving the rest of $E$ fixed. Since $\ell$ does not meet the interior of $\Cone(F)$, there is a half-space $H_+$ with $\ell\subset\partial H_+$ and $H_+\cap \Cone(F)=\emptyset$, so Proposition \ref{p:halfspace} applies. If $\eta(F)>0$, we gain area, which contradicts Corollary \ref{c:halfspace}. If $\eta(F)=0$, the medial axis of $F$ consists of $k$ arcs meeting at a point $p_0$. Close to the point $q = \frac{1}{2} p_0$, the deformed set is a union of $k\geq 4$ half-planes bounded by a line. Two of these half-planes form an angle smaller than $\frac{2\pi}{3}$, so a further local deformation gains area, again contradicting Corollary \ref{c:halfspace}.

Proposition \ref{t:faces} leaves only a problem about stationary nets, which is solved by the following proposition.

\begin{proposition}\label{p:classification}
Let $\N$ be a stationary net with two antipodal pins such that every equiangular face of $\N$ is a biangle or a triangle. Then $\N$ is, up to a rotation, one of the following:
\begin{itemize}
\item[(a)] one half-equator joining the pins;
\item[(b)] two half-equators joining the pins, forming an angle $\geq\frac{2\pi}3$;
\item[(c)] three half-equators joining the pins, forming angles $\frac{2\pi}3$;
\item[(d)] the tetrahedral net, with the pins at the midpoints of two opposite edges.
\end{itemize}
\end{proposition}

The proof, given in Section \ref{s:classification}, is by induction on the number of vertices. Nets with two vertices are elementary, and a pin of degree $3$ is easily excluded. If $\N$ has a face with no pins on its boundary, this face is equiangular, hence it is a biangle or a triangle. Replacing it with its medial axis we obtain a net with fewer vertices and with the same property, which by induction is in the list above; this turns out to be impossible. If every face of $\N$ has a pin on its boundary, a counting argument shows that removing the pins and the edges ending at them leaves a single cycle with $2$, $3$, or $4$ vertices. The cycle with $2$ vertices is easily excluded. The cycles with $3$ and $4$ vertices are treated with a balancing identity for the two pins (Lemma \ref{l:balancing}): the first one would have an equiangular quadrilateral face, and the second one is the tetrahedral net (Lemma \ref{l:four-cycle}).

We can now conclude.

\begin{proof}[Proof of Theorem \ref{t:main}]
Let $\N:=E\cap\partial B_1$. By Proposition \ref{p:stationary}, $\N$ is a stationary net with two antipodal pins, the points of $\ell\cap\partial B_1$. By Proposition \ref{t:faces}, every equiangular face of $\N$ is a biangle or a triangle. By Proposition \ref{p:classification}, $\N$ is, up to a rotation, one of the nets (a)--(d) listed there. Since $E$ is a closed cone, $E=\{tx:\ t\geq0,\ x\in\N\}$, and hence $E$ is the corresponding cone (a)--(d) of Theorem \ref{t:main}. In case (d) the pins are the midpoints of two opposite edges of the tetrahedral net, so $\ell$ is the line through them.

Conversely, as recalled in \cite[Remark 1.5]{DLG}, the cones (a) and (b) are minimizing with respect to $\ell$ by an elementary argument, while the cones (c) and (d) are minimizing with respect to $\emptyset$ by \cite{Taylor}, hence with respect to $\ell$.
\end{proof}

\subsection{Relation with \cite{DLG}} The main novel idea of this paper is contained in Proposition \ref{p:convex-face} and in its use in Proposition \ref{t:faces} (Sections \ref{s:medial} and \ref{s:faces}). They replace most of the analysis of \cite{DLG}, including all its computer-assisted parts. The other ingredients are taken from \cite{DLG}, or are closely inspired by it:
\begin{itemize}
\item Proposition \ref{p:stationary} is \cite[Proposition 4.3]{DLG}.
\item Proposition \ref{p:halfspace} is a particular case of the wedge relaxed condition \cite[Proposition 10.5]{DLG}, and its proof in Section \ref{s:relaxed} is a simplified version of the one in \cite[Section 10]{DLG}.
\item In Section \ref{s:classification} we use the elementary properties of stationary nets with two antipodal pins proved in \cite[Section 5]{DLG}. The nets in which every face has a pin on its boundary are reduced to a cycle with at most four vertices as in \cite[Section 14.3]{DLG}. The cycles with three and four vertices are treated with a new argument based on a balancing identity (Lemmas \ref{l:balancing}--\ref{l:four-cycle}), which replaces \cite[Sections 14.4 and 14.5]{DLG}, including the computer-assisted \cite[Lemma 14.2]{DLG}.
\item The medial axis reduction (Lemma \ref{l:reduction}) is used in the proof of Proposition \ref{p:classification}. The case of biangles is already used in \cite[proof of Proposition 15.8]{DLG}. The case of triangles was also discovered by the human authors, who then omitted it in the manuscript \cite{DLG} because its consequences are covered by other arguments. It is a simple combinatorial counterpart of the idea of Section \ref{s:medial}.
\end{itemize}

\section{Replacing a face with its medial axis}\label{s:medial}

The ideas of this section are similar to known ones, although, as far as we know, Proposition \ref{p:convex-face} and its use in Proposition \ref{t:faces} are new. The vector field $X$ used below, which is constant on each cell of the decomposition of $\Cone(F)$ induced by the medial axis, is in the spirit of the paired calibrations of Lawlor and Morgan \cite{LM}. The map $\Phi$, which pushes the boundary along the inward normals until it reaches the medial axis, is a classical object in the theory of the medial axis (see e.g. \cite{CCM, Merigot}), and identities relating integrals over a boundary to integrals over its medial axis appear in \cite{Damon}.

The construction of Proposition \ref{p:convex-face} is illustrated in Figure \ref{f:medial}.

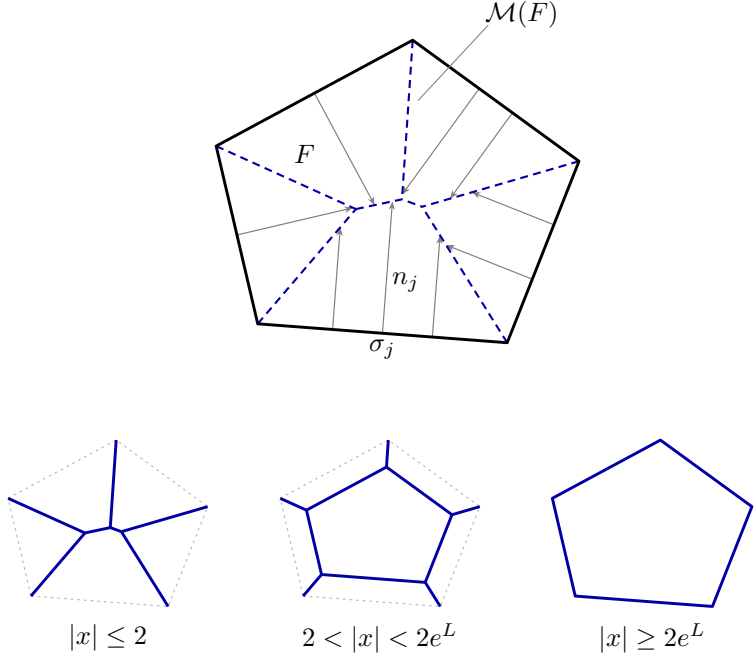
\begin{figure}[ht]
\centering
\begin{tikzpicture}[
  side/.style={line width=1.1pt, line join=round},
  medial/.style={color=blue!65!black, line width=0.8pt, dash pattern=on 3pt off 2pt},
  secm/.style={color=blue!65!black, line width=1.1pt, line join=round},
  ghost/.style={gray!60, line width=0.4pt, dash pattern=on 1pt off 1.5pt},
  push/.style={gray, line width=0.45pt, -{Stealth[length=3.2pt, width=2.4pt]}},
  lab/.style={font=\small, inner sep=1.5pt},
  >={Stealth[length=3.2pt, width=2.4pt]}
]
\begin{scope}[xshift=3.0cm]

  \draw[push] (0.990,-0.075) -- (1.092,1.275);
  \draw[push] (1.650,-0.125) -- (1.782,1.619);
  \draw[push] (2.310,-0.175) -- (2.412,1.169);
  \draw[push] (3.632,0.590) -- (2.492,1.042);
  \draw[push] (3.917,1.310) -- (2.834,1.739);
  \draw[push] (3.370,2.790) -- (2.544,1.655);
  \draw[push] (2.930,3.110) -- (1.912,1.711);
  \draw[push] (0.750,3.050) -- (1.547,1.569);
  \draw[push] (-0.275,1.175) -- (1.260,1.534);
  \draw[medial] (-0.550,2.350) -- (1.299,1.517);
  \draw[medial] (0.000,0.000) -- (1.299,1.517);
  \draw[medial] (3.300,-0.250) -- (2.175,1.548);
  \draw[medial] (4.250,2.150) -- (2.175,1.548);
  \draw[medial] (2.050,3.750) -- (1.908,1.646);
  \draw[medial] (1.299,1.517) -- (1.908,1.646) -- (2.175,1.548);
  \draw[side] (0,0) -- (3.3,-0.25) -- (4.25,2.15) -- (2.05,3.75)
              -- (-0.55,2.35) -- cycle;
  \node[lab] at (0.62,2.25) {$F$};
  \node[lab, below] at (1.65,-0.15) {$\sigma_j$};
  \node[lab, right] at (1.72,0.55) {$n_j$};
  \draw[gray, line width=0.4pt] (3.05,3.95) -- (2.12,2.95);
  \node[lab, above right] at (2.95,3.85) {$\mathcal{M}(F)$};
\end{scope}
\begin{scope}[xshift=0.0cm, yshift=-3.6cm, scale=0.55]
  \draw[ghost] (0,0) -- (3.3,-0.25) -- (4.25,2.15) -- (2.05,3.75) -- (-0.55,2.35) -- cycle;
  \draw[secm] (-0.550,2.350) -- (1.299,1.517);
  \draw[secm] (0.000,0.000) -- (1.299,1.517);
  \draw[secm] (3.300,-0.250) -- (2.175,1.548);
  \draw[secm] (4.250,2.150) -- (2.175,1.548);
  \draw[secm] (2.050,3.750) -- (1.908,1.646);
  \draw[secm] (1.299,1.517) -- (1.908,1.646) -- (2.175,1.548);
  \node[lab] at (1.85,-1.0) {$|x|\le 2$};
\end{scope}
\begin{scope}[xshift=3.6cm, yshift=-3.6cm, scale=0.55]
  \draw[ghost] (0,0) -- (3.3,-0.25) -- (4.25,2.15) -- (2.05,3.75) -- (-0.55,2.35) -- cycle;
  \draw[secm] (0.444,0.518) -- (2.938,0.329) -- (3.582,1.956) -- (2.006,3.102) -- (0.081,2.065) -- cycle;
  \draw[secm] (0.000,0.000) -- (0.444,0.518);
  \draw[secm] (3.300,-0.250) -- (2.938,0.329);
  \draw[secm] (4.250,2.150) -- (3.582,1.956);
  \draw[secm] (2.050,3.750) -- (2.006,3.102);
  \draw[secm] (-0.550,2.350) -- (0.081,2.065);
  \node[lab] at (1.85,-1.0) {$2<|x|<2e^L$};
\end{scope}
\begin{scope}[xshift=7.2cm, yshift=-3.6cm, scale=0.55]
  \draw[secm] (0,0) -- (3.3,-0.25) -- (4.25,2.15) -- (2.05,3.75) -- (-0.55,2.35) -- cycle;
  \node[lab] at (1.85,-1.0) {$|x|\ge 2e^L$};
\end{scope}
\end{tikzpicture}
\caption{Top: a schematic picture of a face $F$ with angles $\frac{2\pi}{3}$ and
its medial axis $\mathcal{M}(F)$ (dashed). Each point of a side $\sigma_j$ is
pushed along the inward normal $n_j$ until it reaches $\mathcal{M}(F)$.
Bottom: the intersection of $\Phi(\Sigma)$ with the sphere $\partial B_{|x|}$,
rescaled, for three values of $|x|$. Close to the origin it is the medial axis.
At intermediate distances it consists of the points at distance
$\psi(|x|)$ from $\partial \Cone(F)$, i.e.\ the intersection with the surface $S$,
joined to the vertices by pieces of $\mathcal{M}(F)$. Far from the origin
nothing moves.}
\label{f:medial}
\end{figure}

\subsection{Proof of Proposition \ref{p:convex-face}}

{\bf Setup.} By the Gauss--Bonnet formula $k\in\{2,3,4,5\}$. In the proof we regard $\Cone(F)$, $\M(F)$, $\Sigma'$ and $\Sigma'_{ij}$ as closed sets, i.e. we replace them with their closures. Note that $\Cone(F)=\bigcap_j\{\langle x,n_j\rangle\ge0\}$. For $x\in \Cone(F)$ set
\[
d(x):=\min_j\langle x,n_j\rangle\, ,
\]
let $V_j:=\{x\in \Cone(F):\ \langle x,n_j\rangle=d(x)\}$, which is a convex polyhedral cone, and let $X:=n_j$ on $V_j$. Note that $\Sigma\cap V_j$ is the cone over $\sigma_j$. If $x\in V_j$, the normal projection $y:=x-d(x)n_j$ of $x$ on the plane of $\sigma_j$ satisfies $\langle y,n_i\rangle\ge d(x)(1-\langle n_j,n_i\rangle)\ge0$ for every $i$, so $y\in\Cone(\sigma_j)$. Since the distance of $x$ from the plane of $\sigma_i$ is $\langle x,n_i\rangle$ and $\partial\Cone(F)=\bigcup_i\Cone(\sigma_i)$, it follows that $d(x)=\dist(x,\partial\Cone(F))$, and that this distance is attained on $\Cone(\sigma_j)$ if and only if $x\in V_j$. For $x\in F$, the sides closest to $x$ in the spherical distance are those whose cones are closest to $x$ in the Euclidean distance, so the same holds for the spherical distance from $x$ to $\partial F$. Consequently $\Sigma'_{ij}=V_i\cap V_j$, which lies in the plane $\{\langle x,n_i-n_j\rangle=0\}$, and $\Sigma'=\bigcup_{i\neq j}\Sigma'_{ij}$. Since two adjacent sides meet at angle $\frac{2\pi}{3}$, the angle between their normals is $\frac{\pi}{3}$, and hence
\begin{equation}\label{e:adjacent}
|n_i-n_j|=1\qquad\text{if $\sigma_i$ and $\sigma_j$ are adjacent.}
\end{equation}

\medskip

{\bf Step 1. The region $\Omega$.} Fix $L\ge2$ and let $\psi$ be the standard logarithmic cutoff
\[
\psi(r):=\begin{cases}2,& 0\le r\le2,\\ 2\big(1-\frac{\log(r/2)}{L}\big),& 2\le r\le2e^L,\\ 0,& r\ge2e^L,\end{cases}
\]
so that $|\psi'(r)|\le\frac{2}{Lr}\le 1$. Set
\[
\Omega:=\{x\in \Cone(F):\ d(x)<\psi(|x|)\}\, ,\qquad S:=\overline{\{x\in \Cone(F):\ d(x)=\psi(|x|)>0\}}\, .
\]
Then $\Omega$ is bounded and $\partial\Omega=(\Sigma\cap B_{2e^L})\cup S$. Since $d(x)\le|x|$, we have $\Cone(F)\cap B_2\subset\Omega$, so $S\cap B_2=\emptyset$ and $\Sigma'\cap B_1\subset\Omega$. On $V_j$ the surface $S$ is the graph $\{\langle x,n_j\rangle=\psi(|x|)\}$ over $n_j^\perp$, with slope at most $1$, and its normal $\nu_S$ pointing out of $\Omega$ is parallel to $n_j-\psi'(|x|)\frac{x}{|x|}$. So $1-\langle X,\nu_S\rangle\le\psi'(|x|)^2$, and
\begin{equation}\label{e:defect}
\int_S\big(1-\langle X,\nu_S\rangle\big)\le C\int_2^{2e^L}\psi'(r)^2\,r\,dr\le\frac{C}{L}\, .
\end{equation}

\medskip

{\bf Step 2. An area identity.} We apply the divergence theorem to the constant field $n_j$ on $V_j\cap\Omega$. On $\Sigma\cap B_{2e^L}\cap V_j$ the outer normal is $-n_j$, and on $\Sigma'_{ij}\cap\Omega$ it is $\frac{n_j-n_i}{|n_j-n_i|}$, whose scalar product with $n_j$ is $\frac12|n_i-n_j|$. Hence
\[
0=-\Ha^2(\Sigma\cap B_{2e^L}\cap V_j)+\int_{S\cap V_j}\langle n_j,\nu_S\rangle+\sum_{i\neq j}\tfrac12|n_i-n_j|\,\Ha^2(\Sigma'_{ij}\cap\Omega)\, .
\]
Summing over $j$ we get
\begin{equation}\label{e:area-identity}
\begin{split}
\Ha^2(S)+\Ha^2(\Sigma'\cap\Omega)={}&\Ha^2(\Sigma\cap B_{2e^L})+\int_S\big(1-\langle X,\nu_S\rangle\big)\\
&-\sum_{i<j}\big(|n_i-n_j|-1\big)\,\Ha^2(\Sigma'_{ij}\cap\Omega)\, .
\end{split}
\end{equation}
In other words, a piece of $\Sigma'_{ij}$ of unit area is reached by two pieces of $\Sigma$ of total area $|n_i-n_j|$, so the collision saves $|n_i-n_j|-1$ per unit area.

\medskip

{\bf Step 3. The coefficients.} We claim that $|n_i-n_j|\geq 1$ whenever $\Ha^2(\Sigma'_{ij})>0$, with strict inequality if $\sigma_i$ and $\sigma_j$ are not adjacent. For biangles and triangles any two sides are adjacent and the claim is \eqref{e:adjacent}. For pentagons and quadrilaterals we use the polar polygon, whose vertices are $n_1, \ldots, n_k$, whose sides measure $\frac{\pi}{3}$, and whose angle at $n_m$ is $\pi-|\sigma_m|$.

If $F$ is a pentagon, any two non-adjacent sides $\sigma_i$ and $\sigma_j$ are separated by a single side $\sigma_m$, and $n_i, n_m, n_j$ are consecutive vertices of the polar polygon. The spherical cosine rule gives
\[
\cos\angle(n_i,n_j)=\tfrac14+\tfrac34\cos(\pi-|\sigma_m|)=\tfrac14-\tfrac34\cos|\sigma_m|\, ,
\]
which is smaller than $\frac12$, because $|\sigma_m|<\lambda_T:=\arccos(-\frac13)$ \cite[Proposition 6.8]{DLG}. Hence $\angle(n_i,n_j)>\frac{\pi}{3}$, i.e. $|n_i-n_j|>1$.

If $F$ is a quadrilateral, the normals $n_1, n_2, n_3, n_4$ are the vertices of the polar polygon, whose sides all measure $\frac{\pi}{3}$. So the polar polygon is a spherical rhombus, which is symmetric with respect to the reflections in its two diagonals. In suitable coordinates $n_{1,3}=(0,\pm u,v)$ and $n_{2,4}=(\pm s,0,t)$, with $u,v,s,t>0$, $u^2+v^2=s^2+t^2=1$, and $vt=\langle n_1,n_2\rangle=\frac12$. Up to relabeling, $v\leq t$. Then $v^2\leq\frac12$, so $|n_1-n_3|=2u\geq\sqrt2>1$. If $v=t$, also $|n_2-n_4|>1$. If $v<t$, note that $x_3>0$ on $\Cone(F)\setminus\{0\}$ and that $V_2\cap V_4$ lies in the plane $\{x_1=0\}$. There
\[
\min\{\langle x,n_1\rangle,\langle x,n_3\rangle\}=vx_3-u|x_2|<tx_3=\langle x,n_2\rangle=\langle x,n_4\rangle\, ,
\]
so $V_2\cap V_4=\{0\}$.

\medskip

{\bf Step 4. The map $\Phi$.} Let $\Phi(x):=x$ for $x\in\Sigma\setminus B_{2e^L}$. For $x\in\Sigma\cap V_j\cap B_{2e^L}$ let $\Phi(x):=x+\tau(x)\,n_j$, where $\tau(x)$ is the first time at which $x+\tau n_j$ hits $\Sigma'\cup S$. Since $d(x+\tau n_j)=\tau$ as long as $x+\tau n_j\in V_j$, $\tau(x)$ is the minimum of the exit time from $V_j$, which is piecewise linear in $x$, and of the solution of $\tau=\psi(|x+\tau n_j|)$, which tends to $0$ as $|x|\to 2e^L$. So $\Phi$ is Lipschitz. The exit time vanishes on the half-lines through the vertices, $\Phi$ moves points by at most $2$, and $\Phi(x)\in V_j$: this gives (1)--(3) with $R:=2e^L+3$. Since $S\cap B_2=\emptyset$, we have $\Phi(\Sigma)\cap B_1=\Sigma'\cap B_1$, which is (4). Finally $\Phi$ maps $\Sigma\cap B_{2e^L}$ into $S\cup(\Sigma'\cap\overline\Omega)$, so by \eqref{e:area-identity} and \eqref{e:defect}
\[
\Ha^2\big(\Phi(\Sigma\cap B_R)\big)\le\Ha^2(\Sigma\cap B_R)+\frac{C}{L}-\sum_{i<j}\big(|n_i-n_j|-1\big)\,\Ha^2(\Sigma'_{ij}\cap\Omega)\, .
\]
By Step 3, and since $\Sigma'\cap B_1\subset\Omega$, the sum is at least $\eta(F)$. Property (5) follows choosing $L$ with $C/L<\delta$.

\medskip

{\bf Step 5. Properties of $\eta(F)$.} Let $W_j:=V_j\cap F$. The $W_j$ are convex spherical polygons which cover $F$ and have pairwise disjoint interiors. By Step 3 and \eqref{e:adjacent}, $\eta (F)\geq 0$ and $\eta(F)>0$ if and only if $W_i\cap W_j$ is a nondegenerate arc for some pair $i,j$ of non-adjacent sides.

If some $p\in F$ is equidistant from all the sides, then $p$ belongs to every $W_j$. Since $W_j$ is convex, it contains the triangle with vertex $p$ and base $\sigma_j$. These triangles tile $F$, so they coincide with the $W_j$, and two of them corresponding to non-adjacent sides meet only at $p$. Hence $\eta(F)=0$.

Conversely, assume $\eta(F)=0$. If $k\leq 3$ the incenter of $F$ (or any point of $\M(F)$, if $k=2$) is equidistant from all the sides, so assume $k\geq 4$. Since $W_j$ meets $W_i$ at most in a point when $\sigma_i$ and $\sigma_j$ are not adjacent, $\partial W_j$ consists of $\sigma_j$ and of the two arcs $W_j\cap W_{j\pm 1}$, which start at the endpoints of $\sigma_j$ and meet at a point $p_j$. The arc $W_j\cap W_{j+1}$ ends at $p_j$ as part of $\partial W_j$ and at $p_{j+1}$ as part of $\partial W_{j+1}$, hence $p_j=p_{j+1}$. So all the $p_j$ coincide, and this point is equidistant from all the sides. \qed

\section{Weakening the Almgren minimizing condition}\label{s:relaxed}

In this section we prove Proposition \ref{p:halfspace} and Corollary \ref{c:halfspace}. Proposition \ref{p:halfspace} is a particular case of the wedge relaxed condition \cite[Proposition 10.5]{DLG}; we include a direct proof, which is simpler in this case.

\begin{figure}
\centering
\begin{tikzpicture}[scale=0.72,
  net/.style={line width=1.1pt, line join=round, line cap=round},
  moved/.style={color=blue!65!black, line width=1.1pt, line join=round, line cap=round},
  ghost/.style={gray!70, line width=0.4pt, dash pattern=on 1pt off 1.5pt},
  circ/.style={gray, line width=0.4pt, dash pattern=on 2pt off 1.5pt},
  bdry/.style={gray, line width=0.4pt},
  half/.style={fill=gray!12},
  lab/.style={font=\small, inner sep=1.5pt},
  slab/.style={font=\footnotesize, inner sep=1pt},
  >={Stealth[length=3.2pt, width=2.4pt]}
]
\def\background{%
  \fill[half] (-3.1,0) rectangle (3.1,3.05);
  \draw[bdry] (-3.1,0) -- (3.1,0);
}
\def\ellpoint{\fill[red] (0,0) circle (2.1pt);}

\begin{scope}[xshift=0cm]
  \background
  \draw[circ] (0,0) circle (0.6);
  \draw[circ] (0,0) circle (1.2);
  \draw[net] (0,0) -- (0.776,2.898);
  \draw[net] (0,0) -- (-2.719,-1.268);
  \draw[net] (0,0) -- (2.719,-1.268);
  \draw[gray, line width=0.4pt] (0,0) -- (150:1.2);
  \node[slab, anchor=south east] at (150:1.2) {$2\rho$};
  \node[slab, anchor=south east] at (150:0.6) {$\rho$};
  \draw[->, line width=0.6pt] (-2.4,0) -- (-2.4,0.85);
  \node[lab, right] at (-2.4,0.55) {$v$};
  \node[lab] at (-1.9,2.4) {$H_+$};
  \node[lab] at (0,-0.27) {$\ell$};
  \ellpoint
  \node[lab] at (0,-3.0) {(a) $E$};
\end{scope}

\begin{scope}[xshift=6.6cm]
  \background
  \draw[ghost] (0,0) -- (-2.719,-1.268);
  \draw[ghost] (0,0) -- (2.719,-1.268);
  \draw[net] (0,0) -- (0.776,2.898);
  \draw[moved] (0,0) -- (0,-1.6);
  \draw[moved] (-2.719,-1.268) -- (0,-1.6) -- (2.719,-1.268);
  \ellpoint
  \node[lab] at (0,-3.0) {(b) $\Psi(E)$};
\end{scope}

\begin{scope}[xshift=13.2cm]
  \background
  \draw[circ] (0,0) circle (0.6);
  \draw[circ] (0,0) circle (1.2);
  \draw[ghost] (0.155,0.580) -- (0.776,2.898);
  \draw[ghost] (0,0) -- (0,-1.6);
  \draw[ghost] (-2.719,-1.268) -- (0,-1.6) -- (2.719,-1.268);
  \draw[bdry, dash pattern=on 0.8pt off 1.2pt] (-3.1,-0.6) -- (3.1,-0.6);
  \node[font=\scriptsize, inner sep=1pt, anchor=north east] at (3.08,-0.70) {$\langle x,v\rangle=-\rho$};
  \draw[net] (0,0) -- (0.155,0.580);
  \draw[net] (0,0) -- (-0.544,-0.254);
  \draw[net] (0,0) -- (0.544,-0.254);
  \draw[moved] (0.155,0.580) -- (0.168,0.578) -- (0.181,0.576) -- (0.194,0.574) -- (0.207,0.573) -- (0.220,0.571) -- (0.233,0.569) -- (0.246,0.568) -- (0.259,0.566) -- (0.272,0.564) -- (0.285,0.563) -- (0.298,0.561) -- (0.311,0.559) -- (0.776,2.298);
  \draw[moved] (-0.544,-0.254) -- (-0.540,-0.347) -- (-0.529,-0.445) -- (-0.510,-0.546) -- (-0.483,-0.650) -- (-0.449,-0.758) -- (-0.408,-0.869) -- (-0.359,-0.984) -- (-0.302,-1.102) -- (-0.238,-1.224) -- (-0.166,-1.349) -- (-0.087,-1.478) -- (0.000,-1.611) -- (0.000,-2.200) -- (-2.719,-1.868);
  \draw[moved] (0.544,-0.254) -- (0.540,-0.347) -- (0.529,-0.445) -- (0.510,-0.546) -- (0.483,-0.650) -- (0.449,-0.758) -- (0.408,-0.869) -- (0.359,-0.984) -- (0.302,-1.102) -- (0.238,-1.224) -- (0.166,-1.349) -- (0.087,-1.478) -- (0.000,-1.611) -- (0.000,-2.200) -- (2.719,-1.868);
  \draw[->, gray, line width=0.5pt] (-1.8,-1.42) -- (-1.8,-1.94);
  \node[slab, left] at (-1.82,-1.66) {$-\rho v$};
  \ellpoint
  \node[lab] at (0,-3.0) {(c) $G(E)$};
\end{scope}
\end{tikzpicture}
\caption{The construction in the proof of Proposition \ref{p:halfspace}, in a cross-section by a plane orthogonal to $\ell$ (the red dot) where $\chi=1$. (a) The set $E$, the half-space $H_+$ (shaded) and the circles $d=\rho$ and $d=2\rho$. (b) $\Psi$ fixes $E\cap H_+$, but the moved part (blue) touches $\ell$. (c) $G$ is the identity where $d\le\rho$ and equals $\Psi-\rho v$ where $d\ge 2\rho$ ($\Psi(E)$ is dotted), so the moved part of $G(E)$ (blue) stays away from $\ell$. In $H_+$ the map $G$ fixes the points at distance at most $\rho$ from $\ell$, and moves the other points by at most $\rho$ in the direction $-v$, keeping them at distance at least $\rho$ from $\ell$.}\label{f:halfspace}
\end{figure}
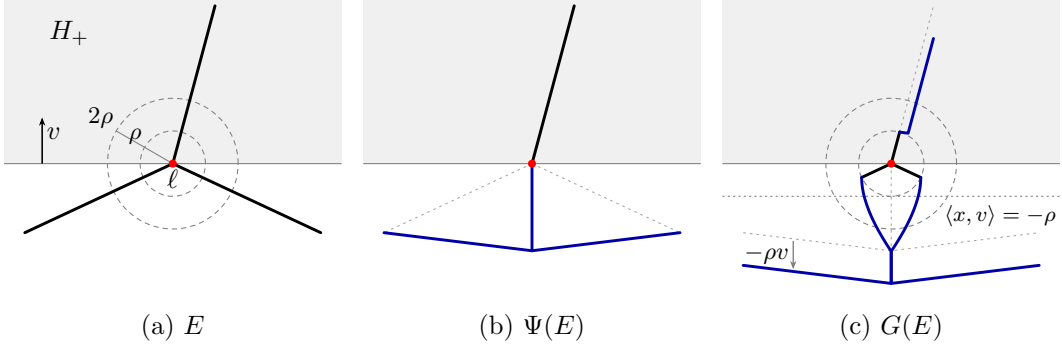

\begin{proof}[Proof of Proposition \ref{p:halfspace}]
By \cite[Proposition 10.1]{DLG} it suffices to define $G$ on $E$, with $G(x)=x$ on $E\setminus B_1$ and $\overline{W\cup G(W)}\cap \ell=\emptyset$.

Let $v$ be the unit vector such that $H_+=\{\langle x,v\rangle>0\}$, and note that $v\perp\ell$. Let $d(x):=\dist(x,\ell)$ and $L:=\Lip(\Psi)$. Since $\Psi$ is the identity on $\ell$, we have $|\Psi(x)-x|\le(L+1)d(x)$. Fix a small $\rho>0$ and let
\[
t(x):=\min\Big\{1,\tfrac{(d(x)-\rho)^+}{\rho}\Big\},\qquad \chi(x):=\min\Big\{1,\tfrac{(1-|x|)^+}{\rho}\Big\}\, .
\]
For $x\in E$ we define
\[
G(x):=x+t(x)\big(\Psi(x)-x-\rho\chi(x)v\big)\, .
\]
Away from a small set, $G$ is $\Psi$ followed by the translation by $-\rho v$ (see Figure \ref{f:halfspace}), which pushes the moved points away from $\ell$; the cutoffs $t$ and $\chi$ restore the identity near $\ell$ and near $\partial B_1$.
The Lipschitz constant of $G$ is bounded independently of $\rho$: where $\nabla t\neq0$ we have $d\le2\rho$ and so $|\Psi(x)-x-\rho\chi v|\le(2L+3)\rho$, while $|\nabla\chi|\le \rho^{-1}$ is multiplied by $\rho t\leq \rho$. Clearly $G(x)=x$ if $|x|\geq 1$ or $d(x)\le\rho$. Hence $\overline W$ is a compact subset of $\{|x|\le1,\ d(x)\ge\rho\}$ and $\overline{G(W)}\subset G(\overline W)$. So it suffices to show that $G(x)\notin\ell$ whenever $|x|\leq 1$ and $d(x)\geq\rho$:
\begin{itemize}
\item if $x\in H_+$, then $\Psi (x) = x$ and $G(x)=x-t\rho\chi v$, so $d(G(x))\ge d(x)-t\rho\ge\rho$, because $d(x)\geq\rho(1+t(x))$;
\item if $x\not\in H_+$, then $\Psi(x)\not\in H_+$ by (iii), and thus $\langle G(x),v\rangle\le-t\rho\chi$. If $t\chi>0$, then $G(x)\notin\ell$. Otherwise $G(x)=x$: this is clear if $t=0$, while if $\chi=0$ then $|x|=1$ and $\Psi(x)=x$.
\end{itemize}

As for the area, on $A_1:=\{x\in E:\ d(x)\geq 2\rho,\ |x|\leq 1-\rho\}$ the map $G$ is the translation of $\Psi$ by $-\rho v$, so $\Ha^2(G(A_1))\le\Ha^2(\Psi(E\cap\overline B_1))$. The remaining part $A_2:=(E\cap\overline B_1)\setminus A_1$ has area at most $C\rho$, and the Lipschitz constant of $G$ is bounded independently of $\rho$, hence $\Ha^2(G(A_2))\le C\rho$. Choosing $\rho$ small we get (c).
\end{proof}

\begin{proof}[Proof of Corollary \ref{c:halfspace}]
Let $G$ be given by Proposition \ref{p:halfspace}. By (a) and (b), $G$ is admissible in Definition \ref{d:Almgren-minimizing} with $V$ a large ball. Since $G(x)=x$ for $x\in E\setminus B_1$, minimality gives $\Ha^2(E\cap\overline B_1)\le\Ha^2(G(E\cap\overline B_1))\le\Ha^2(\Psi(E\cap\overline B_1))+\eps$. Let $\eps\downarrow 0$.
\end{proof}

\section{Faces with angles \texorpdfstring{$\frac{2\pi}{3}$}{2pi/3}}\label{s:faces}

In this section we prove Proposition \ref{t:faces}, combining Propositions \ref{p:halfspace} and \ref{p:convex-face}.

\begin{proof}[Proof of Proposition \ref{t:faces}]
Let $\N := E\cap \partial B_1$ and let $F$ be an equiangular face of $\N$. $F$ is weakly locally convex, hence convex by \cite[Lemma B.1]{DLG}, and its boundary is a Jordan curve. So $F$ is a convex spherical polygon with all angles $\frac{2\pi}{3}$, and by Gauss--Bonnet it has $k\in\{2,3,4,5\}$ vertices. Assume by contradiction $k\in\{4,5\}$ and let $\Sigma$ and $\Sigma'$ be as in Proposition \ref{p:convex-face}.

\medskip

{\bf Step 1. A half-space.} The pins belong to $\N$, so $\ell$ does not meet the interior of $\Cone(F)$, which is an open convex set. By the Hahn--Banach theorem there are a nonzero linear function $f$ and $c\in\R$ with $f\le c$ on $\ell$ and $f>c$ on the interior of $\Cone(F)$. Since $\ell$ is a line, $f=0$ on $\ell$ and $c\ge0$. Hence $H_+:=\{f<0\}$ is an open half-space with $\ell\subset\partial H_+$ and $\Cone(F)\cap H_+=\emptyset$.

\medskip

{\bf Step 2. The deformation.} For $\delta>0$ let $R$ and $\Phi_\delta$ be given by Proposition \ref{p:convex-face}, and extend $\Phi_\delta$ to $E$ setting $\Phi_\delta (x)=x$ on $E\setminus\Sigma$. Since $F$ is a face, $E\cap \Cone(F)=\Sigma$. The extension is Lipschitz, because $\Sigma$ meets $\overline{E\setminus \Sigma}$ only along the half-lines through the vertices of $F$, where $\Phi_\delta$ is the identity, and the edges of $\N$ meet at positive angles. Recall that $\N$ is connected \cite[Proposition 5.1(ii)]{DLG}. After the rescaling $x\mapsto R^{-1}\Phi_\delta(Rx)$, the map satisfies the assumptions of Proposition \ref{p:halfspace}: points of $\ell$ are fixed by property (1), and both the set of moved points and its image are contained in $\Cone(F)$, which does not intersect $H_+$.

\medskip

{\bf Step 3. The case $\eta(F)>0$.} Choose $\delta:=\frac{1}{2}\eta(F)$. By property (5),
\[
\Ha^2\big(\Phi_\delta(E\cap B_R)\big)\le\Ha^2(E\cap B_R)-\tfrac12\eta(F)\, ,
\]
which contradicts Corollary \ref{c:halfspace}.

\medskip

{\bf Step 4. The case $\eta(F)=0$.} By Proposition \ref{p:convex-face} and Step 5 of its proof, $\M(F)$ consists of $k$ arcs joining a point $p_0\in F$ to the vertices of $F$. Hence, in a small ball $B\subset B_1$ centered at $q:=\frac12 p_0$ and contained in the interior of $\Cone(F)$, the set $\Sigma'$ consists of $k\geq4$ half-planes bounded by the line $\R p_0$, two of which form an angle smaller than $\frac{2\pi}3$. Such a cone is not minimizing in any ball centered at its vertex \cite[Lemma 4.7 and the subsequent discussion]{DLG}: there are $\eps_1>0$ and a Lipschitz map $\Lambda$ which is the identity outside $B$, maps $B$ into itself, and decreases the area of $\Sigma'\cap B$ by $\eps_1$. Choose $\delta<\eps_1$. By property (4), $\Phi_\delta(E\cap B_R)\cap B=\Sigma'\cap B$, so by property (5)
\[
\Ha^2\big(\Lambda(\Phi_\delta(E\cap B_R))\big)\le\Ha^2(E\cap B_R)+\delta-\eps_1<\Ha^2(E\cap B_R)\, .
\]
Since $\Lambda$ only moves points inside $B\subset\Cone(F)$, the map $\Lambda\circ\Phi_\delta$ satisfies the assumptions of Proposition \ref{p:halfspace} as in Step 2, and this contradicts Corollary \ref{c:halfspace}.
\end{proof}

\section{Classification of nets without equiangular quadrilaterals and pentagons}\label{s:classification}

Throughout this section $\N$ is a stationary net with two antipodal pins. We use the following facts from \cite{DLG}:
\begin{itemize}
\item $\N$ is connected, and every face is an open disk \cite[Proposition 5.1(ii)]{DLG};
\item if $\N$ has at least three vertices, no face has both pins on its boundary \cite[Proposition 5.1(iv)]{DLG};
\item an equiangular face is convex \cite[Lemma B.1]{DLG}. In particular it lies on one side of each edge of its boundary, and its boundary visits each of its vertices once.
\end{itemize}
Since an edge ending at a pin $P$ cannot contain $-P$ in its interior, if $\N$ has at least three vertices every edge ending at a pin is shorter than $\pi$. An equiangular biangle is bounded by two half-equators joining two antipodal points, and an equiangular triangle is equilateral with side $\lambda_T$ \cite[Proposition 6.8]{DLG}.

\subsection{Medial axis reduction}

\begin{lemma}[Medial axis reduction]\label{l:reduction}
Let $\N$ be a stationary net with two antipodal pins such that every equiangular face of $\N$ is a biangle or a triangle. Let $T$ be an equiangular face of $\N$ which is a biangle or a triangle and none of whose vertices is a pin. Let $\tilde{\N}$ be the net obtained by replacing the sides of $T$ with its medial axis. Then:
\begin{itemize}
\item[(1)] $\tilde{\N}$ is a stationary net with the same two antipodal pins, and it has two vertices fewer than $\N$;
\item[(2)] every equiangular face of $\tilde{\N}$ is a biangle or a triangle.
\end{itemize}
\end{lemma}

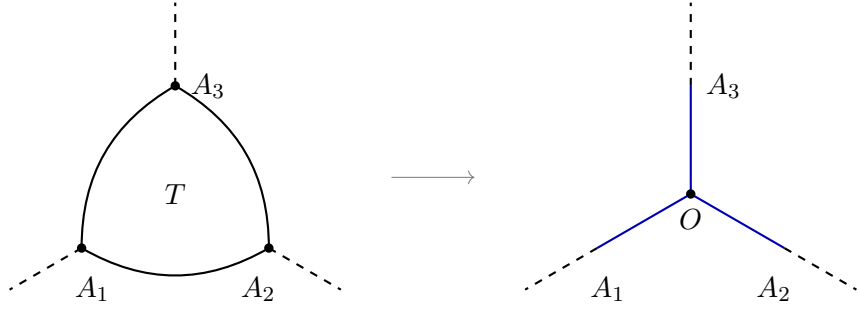
\begin{figure}
\centering
\begin{tikzpicture}[scale=1.1]
\def\rT{1.3}\def\lT{1.0}
\begin{scope}
\coordinate (O) at (0,0);
\foreach \i/\a in {1/210,2/330,3/90}{
  \coordinate (A\i) at (\a:\rT);
  \coordinate (B\i) at (\a:\rT+\lT);
}
\draw[thick] (A1) to[bend right=30] (A2) to[bend right=30] (A3) to[bend right=30] (A1);
\foreach \i in {1,2,3}{\draw[thick,dashed] (A\i) -- (B\i); \fill (A\i) circle (1.6pt);}
\node at (0,0) {$T$};
\node[below=7pt,xshift=4pt] at (A1) {$A_1$};
\node[below=7pt,xshift=-4pt] at (A2) {$A_2$};
\node[right=2pt] at (A3) {$A_3$};
\end{scope}
\draw[->,gray] (2.6,0.2) -- (3.6,0.2);
\begin{scope}[xshift=6.2cm]
\coordinate (O) at (0,0);
\foreach \i/\a in {1/210,2/330,3/90}{
  \coordinate (A\i) at (\a:\rT);
  \coordinate (B\i) at (\a:\rT+\lT);
}
\foreach \i in {1,2,3}{\draw[thick,blue!70!black] (O) -- (A\i); \draw[thick,dashed] (A\i) -- (B\i);}
\fill (O) circle (1.6pt);
\node[below=2pt] at (O) {$O$};
\node[below=7pt,xshift=4pt] at (A1) {$A_1$};
\node[below=7pt,xshift=-4pt] at (A2) {$A_2$};
\node[right=2pt] at (A3) {$A_3$};
\end{scope}
\end{tikzpicture}
\caption{The medial axis reduction of Lemma \ref{l:reduction} when $T$ is a triangle. Left: $T$ and the other edges at its vertices (dashed). Right: the sides of $T$ are replaced by its medial axis, namely the three arcs joining the vertices to the center $O$, which extend the dashed edges.}\label{f:reduction}
\end{figure}

\begin{proof}
For a vertex $V$ of $T$, let $e_V$ be the edge of $\N$ at $V$ which is not a side of $T$. It forms angles $\frac{2\pi}3$ with both sides, so it extends the arc of the medial axis ending at $V$ (see Figure \ref{f:reduction}).

For (1):
\begin{itemize}
\item if $T$ is a triangle, its medial axis consists of three arcs from its center $O$, forming angles $\frac{2\pi}3$ at $O$, and (1) is clear;
\item if $T$ is a biangle with vertices $\pm U$, its medial axis is the half-equator $\mu$ bisecting $T$, and we only need that $e_U\cup\mu\cup e_{-U}$ is embedded. Otherwise either $e_U=e_{-U}$, and then $\pm U$ and the three edges joining them form a component of $\N$ without pins; or $e_U$ and $e_{-U}$ meet at a point $W$ forming an angle $\pi$, so $W$ is a pin and $-W$ lies inside $T$.
\end{itemize}

For (2), let $\tilde F$ be an equiangular face of $\tilde\N$ which meets $T$. Since $\tilde F$ is convex, it consists of a piece of $T$, a side $\sigma$ of $T$, and the face $F$ of $\N$ on the other side of $\sigma$. The angles of $F$ away from the endpoints of $\sigma$ are angles of $\tilde F$, so $F$ is equiangular, hence a biangle or a triangle:
\begin{itemize}
\item if $T$ is a triangle, then $|\sigma|=\lambda_T\neq\pi$, so $F$ is a triangle and $\tilde F$ is a biangle;
\item if $T$ is a biangle, then $|\sigma|=\pi$, so $F$ is a biangle with vertices $\pm U$ and second side $e_U=e_{-U}$, which we have excluded.\qedhere
\end{itemize}
\end{proof}

\subsection{Balancing and four-cycles} For a vertex $V$ of a geodesic net, let $S(V)$ be the sum of the unit tangent vectors at $V$ of the edges ending at $V$. If the net is stationary and $V$ is not a pin, then $S(V)=0$.

\begin{lemma}[Balancing]\label{l:balancing}
If $\N$ is a stationary net with two antipodal pins $P$ and $-P$, then $S(P)=S(-P)$.
\end{lemma}

\begin{proof}
If $\gamma$ is an edge parametrized by arclength, then $\gamma\times\gamma'$ is constant. Hence, if $a$ and $b$ are the endpoints of $\gamma$ and $\tau_a$, $\tau_b$ are the unit tangent vectors of $\gamma$ at $a$ and $b$ pointing into $\gamma$, we have $a\times\tau_a+b\times\tau_b=0$. Summing over all the edges we get $\sum_V V\times S(V)=0$, that is $P\times\big(S(P)-S(-P)\big)=0$. Since $S(P)$ and $S(-P)$ are orthogonal to $P$, they coincide.
\end{proof}

\begin{lemma}[Biangles]\label{l:lune}
Let $L\subset\Sf$ be a closed biangle with vertices $\pm P$ and angle smaller than $\pi$, i.e.\ the closed region bounded by two half great circles joining $P$ and $-P$ and forming an angle smaller than $\pi$, and let $\N$ be a geodesic net whose vertices lie in $L$.
\begin{itemize}
\item[(i)] Let $V\neq\pm P$ be a vertex of $\N$ on $\partial L$ where three edges meet with angles $\frac{2\pi}{3}$, one of which is contained in $\partial L$. Then $V$ is an endpoint of an edge longer than $\pi$.
\item[(ii)] At most one edge of $\N$ which does not end at $\pm P$ is longer than $\pi$.
\end{itemize}
\end{lemma}

\begin{proof}
The set $L\setminus\{\pm P\}$ is contained in an open hemisphere $H$, so it contains no pair of antipodal points.

(i) One of the edges at $V$ leaves $V$ towards the exterior of $L$, and its other endpoint lies in $L$. Since $L$ is convex, this edge is longer than $\pi$.

(ii) An edge longer than $\pi$ with endpoints in $H$ contains the closed half of its great circle which lies outside $H$. Two such closed half great circles intersect, because two great circles meet at two antipodal points, one of which lies outside $H$. Since the intersection point is not a vertex, two such edges cannot exist.
\end{proof}

In \cite[Section 14.5]{DLG} the following lemma is proved with the help of a computer (cf. \cite[Lemma 14.2]{DLG}). We give here a proof without computer assistance.

\begin{lemma}\label{l:four-cycle}
Let $\N$ be a stationary net with two antipodal pins $P$ and $-P$ of degree $2$, and assume that removing the pins and the edges ending at them we obtain a cycle $\N_0$ with four vertices, each of which is joined to exactly one pin. Then $\N$ is the tetrahedral net, with the pins at the midpoints of two opposite edges.
\end{lemma}

\begin{proof}

\medskip

{\bf Step 1. The pins are flat.} Assume by contradiction that $S(P)\neq0$, i.e. that the two edges at $P$ form an angle $\theta<\pi$. Two unit vectors orthogonal to $P$ are determined, up to order, by their sum. Hence, by Lemma \ref{l:balancing}, the edges at $-P$ have the same unit tangent vectors $v_1,v_2$ as the edges at $P$. The half great circle leaving $P$ with tangent $v_i$ also leaves $-P$ with tangent $v_i$. So the four edges ending at the pins lie on the two sides of a closed biangle $L$ with vertices $\pm P$ and angle $\theta$. These edges are shorter than $\pi$, so the four vertices of $\N_0$ lie in the interior of the sides of $L$. By Lemma \ref{l:lune}(i) each of them is an endpoint of an edge longer than $\pi$, which does not end at a pin. So there are at least two such edges, contradicting Lemma \ref{l:lune}(ii).

\medskip

{\bf Step 2. Conclusion.} By Step 1, at each pin the two edges form a single geodesic arc. Erasing the pins we obtain a geodesic net $\N'$ with four vertices, all of degree $3$ with angles $\frac{2\pi}{3}$, and six edges. $\N'$ has no loops, and two vertices $U$ and $V$ joined by two edges are antipodal: otherwise the two edges would lie on the great circle through $U$ and $V$ and form an angle $\pi$ at $U$. Assume that $U$ and $-U$ are joined by two edges. These are half great circles forming an angle $\frac{2\pi}3$, so the third edges at $U$ and at $-U$ lie on the same open half great circle $\mu$ joining $U$ and $-U$. Their other endpoints are the two remaining vertices $W_1\neq W_2$, and counting degrees $W_1$ and $W_2$ are joined by two edges. So they are antipodal, which is impossible because they lie on $\mu$. Hence $\N'$ is the complete graph on four vertices, and its faces are four triangles with angles $\frac{2\pi}{3}$. These are equilateral with side $\lambda_T$, so $\N'$ is the tetrahedral net. Finally, the image of the tetrahedral net under the antipodal map is the dual tetrahedral net, which meets it exactly at the midpoints of the six edges. So $P$ and $-P$ are the midpoints of two opposite edges.
\end{proof}

\subsection{Proof of Proposition \ref{p:classification}}
We argue by induction on the number of vertices of $\N$.

\medskip

{\bf Step 1. Two vertices.} By \cite[Proposition 5.1(iii)]{DLG} $\N$ consists of one, two, or three half-equators joining the pins. Stationarity at the pins gives (a), (b), or (c).

From now on $\N$ has at least three vertices. In particular no equiangular biangle has a pin as a vertex: otherwise both its vertices would be pins.

\medskip

{\bf Step 2. No pin has degree $3$.} Assume $P$ has degree $3$, with neighbors $A_1,A_2,A_3$. The three faces at $P$ have angle $\frac{2\pi}3$ at $P$ and no other pin on their boundary, so they are equiangular. They are not biangles, so they are triangles. These equilateral triangles share the edges $PA_i$, so $P,A_1,A_2,A_3$ are the vertices of a tetrahedral net and the edges $A_iA_j$ belong to $\N$. By connectedness $\N$ is this tetrahedral net. But then $-P$, the center of the face $A_1A_2A_3$, does not belong to $\N$, a contradiction.

Hence both pins have degree $1$ or $2$.

\medskip

{\bf Step 3. Some face has no pin on its boundary.} We show that this case cannot occur. Let $T$ be such a face. All its angles are at vertices which are not pins, so $T$ is equiangular, and by assumption it is a biangle or a triangle.

By Lemma \ref{l:reduction}, the net $\tilde\N$ obtained replacing the sides of $T$ with its medial axis is a stationary net with the same two antipodal pins, with fewer vertices, and whose equiangular faces are biangles or triangles. By the induction hypothesis $\tilde\N$ is one of the nets (a)--(d). This is impossible:
\begin{itemize}
\item if $T$ is a triangle, the edges of $\tilde\N$ at the center $O$ of $T$ contain the arcs joining $O$ to the vertices of $T$, which are longer than half the side of $T$, i.e.\ than $\frac{\lambda_T}2$. But the nets (a)--(c) have no vertex other than the pins, and in (d) every vertex which is not a pin is the endpoint of an edge of length $\frac{\lambda_T}2$;
\item if $T$ is a biangle, $\tilde\N$ has an edge which contains the half-equator bisecting $T$, hence longer than $\pi$, while no net in (a)--(d) has such an edge.
\end{itemize}

\medskip

{\bf Step 4. Every face has a pin on its boundary.} Let $n\geq 1$ be the number of vertices which are not pins, and let $D$ be the sum of the degrees of the two pins.

No edge joins the two pins, no vertex is joined to both pins, and no vertex is joined to the same pin by two edges: in the first two cases some face would have both pins on its boundary, and in the third case the two edges would lie on the same great circle. Hence $D\leq n$. On the other hand, every face contains a sector at one of the pins, so the number $f$ of faces is at most $D$. Since $\N$ is connected, Euler's formula gives $f=2-(n+2)+\frac{3n+D}{2}=\frac{n+D}{2}$. So $n\leq D$, and therefore $n=D\in\{2,3,4\}$, and every vertex which is not a pin is joined to exactly one pin.

Let $\N_0$ be the net obtained removing the pins and the edges ending at them. All its vertices have degree $2$, so $\N_0$ is a union of $c\geq 1$ disjoint closed polygonal curves, and it has $c+1$ faces. Each face of $\N$ touches a pin, so each face of $\N_0$ contains a pin. Hence $c=1$: $\N_0$ is a closed polygon with $n$ vertices, and its two faces $Q$ and $Q'$ contain $P$ and $-P$ respectively. At each vertex of $\N_0$ the angle is $\frac{4\pi}3$ in the face containing the pin to which the vertex is joined, and $\frac{2\pi}3$ in the other face.

\begin{itemize}
\item $n=2$. The two vertices of $\N_0$ are antipodal and $Q$, $Q'$ are biangles. The two angles of a biangle are equal, so one of the two faces, say $Q$, has both angles equal to $\frac{2\pi}3$. Then no vertex is joined to $P$, a contradiction.
\item $n=3$. The pins have degrees $1$ and $2$: up to exchanging them, $P$ has degree $2$, with neighbors $B$ and $C$, and $-P$ has degree $1$, with neighbor $A$. By Lemma \ref{l:balancing} the sum of the unit tangent vectors at $P$ of the edges $PB$ and $PC$ is the unit tangent vector at $-P$ of the edge $(-P)A$. Hence the edges at $P$ form an angle $\frac{2\pi}{3}$ and, as in Step 1 of the proof of Lemma \ref{l:four-cycle}, $B$ and $C$ lie in the interior of the sides of a closed biangle $L$ with vertices $\pm P$ and angle $\frac{2\pi}3$, while $A$ lies inside $L$ on its bisector. By Lemma \ref{l:lune}, $B$ and $C$ are endpoints of the same edge longer than $\pi$, namely $BC$, while the other two edges $AB$ and $AC$ of $\N_0$ are shorter than $\pi$, hence contained in $L$. So the quadrilateral $PBAC$ is a face of $\N$, and all its angles are $\frac{2\pi}3$ (the edge $(-P)A$ lies on the other side of $BAC$). This contradicts the assumption.
\item $n=4$. Both pins have degree $2$, and by Lemma \ref{l:four-cycle} $\N$ is the net (d). \qed
\end{itemize}

\end{document}